\documentclass[a4]{article}
\usepackage{amssymb,amsmath,amsbsy,amsthm}
\usepackage{color,a4wide}
\usepackage{graphicx,hyperref}
\usepackage{booktabs}
\usepackage{multirow}
\newcommand{\norm}[1]{\| #1 \|}
\newcommand{\sprod}[2]{\langle #1, #2 \rangle}
\theoremstyle{plain}
\newtheorem{theorem}{Theorem}
\newtheorem{lemma}[theorem]{Lemma}
\theoremstyle{definition}

\title{Robust Multigrid for Isogeometric Analysis Based on Stable Splittings of Spline Spaces}

\author{Clemens Hofreither \footnote{The first author was supported by the
  National Research Network ``Geometry + Simulation'' (NFN S117, 2012--2016),
  funded by the Austrian Science Fund (FWF)}\\Department of Computational Mathematics,\\
  Johannes Kepler University Linz, Austria\\
  {chofreither@numa.uni-linz.ac.at}
  \and
  Stefan Takacs\\Johann Radon Institute for Computational and
  Applied Mathematics (RICAM),\\Austrian Academy of Sciences, Austria\\
  {stefan.takacs@ricam.oeaw.ac.at}
}

\begin{document}

\maketitle

\begin{abstract}
We present a robust and efficient multigrid method for single-patch isogeometric
discretizations using tensor product B-splines of maximum smoothness.  Our
method is based on a stable splitting of the spline space into a large subspace
of ``interior'' splines which satisfy a robust inverse inequality, as well as
one or several smaller subspaces which capture the boundary effects responsible
for the spectral outliers which occur in Isogeometric Analysis.  We then
construct a multigrid smoother based on an additive subspace correction approach,
applying a different smoother to each of the subspaces.  For the interior
splines, we use a mass smoother, whereas the remaining components are treated
with suitably chosen Kronecker product smoothers or direct solvers.

We prove that the resulting multigrid method exhibits iteration numbers which
are robust with respect to the spline degree and the mesh size.
Furthermore, it can be efficiently realized for discretizations of
problems in arbitrarily high geometric dimension.
Some numerical examples illustrate the theoretical results and show that the
iteration numbers also scale relatively mildly with the problem dimension.
\end{abstract}

%%%%%%%%%%%%%%%%%%%%%%%%%%%%%%%%%%%%%%%%%%%%%%%%%%%%%%%%%%%%%%%%%%%%%%%%%%%%%%%%

\section{Introduction}
\label{sec:intro}

Isogeometric Analysis (IgA) is a method for the numerical solution of partial
differential equations (PDEs) introduced in the seminal paper
\cite{Hughes:2005} which has since attracted a sizable research community.
Spline spaces, such as spaces spanned by tensor product B-splines or NURBS, are
commonly used for geometry representation in industrial CAD systems.  The
foundational idea in IgA is to use such spline spaces both for the
representation of the computational domain and for the discretization of the
quantities of interest when solving a PDE.  The overall goal is to create a
tighter integration between geometric design and analysis.

There is a need for efficient solvers for the large, sparse linear
systems which arise when applying isogeometric discretizations to boundary
value problems.
By now, most established solution strategies known from the finite element
literature have been applied in one way or another to IgA:
among these, direct solvers \cite{CollierEtAl:2012},
non-overlapping and overlapping domain decomposition methods
\cite{KleissEtAl:2012,DaVeigaEtAl:2012,DaVeigaEtAl:2013,DaVeigaEtAl:2014},
and multilevel and multigrid methods
\cite{BuffaEtAl:2013,GahalautEtAl:2013,HofreitherZulehner:2014c,Donatelli:2014a,HTZ:2016}.
A recent contribution \cite{Sangalli:2016} constructs preconditioners based on
fast solvers for Sylvester equations.
The above list is certainly not comprehensive.

In IgA, we typically encounter as discretization parameters the mesh size and
the spline degree.  In the early IgA solver literature, the focus was on
translating solvers from the finite element world to IgA with minimal
adaptations. As a rule, it was found that such an approach results in methods
that work well for low spline degrees, but deteriorate in performance as the
degree is increased; often dramatically so.  This motivated the search for IgA
solvers that are robust not only with respect to the mesh size (which is often
easy to achieve), but also with respect to the spline degree.

Within the class of multigrid methods for IgA, advances towards a robust method
were made using two approaches.  In \cite{Donatelli:2014b}, a careful analysis
of the symbol of isogeometric stiffness matrices served as the basis for the
construction of multigrid methods. This theoretical approach is somewhat
related to the technique known as Local Fourier Analysis (LFA) in the multigrid
literature (see, e.g., \cite{Trottenberg:2001}).  It appears that the method
presented in \cite{Donatelli:2014b} is roughly comparable to the one studied in
\cite{HofreitherZulehner:2014b}, which uses mass matrices as multigrid
smoothers, an approach itself motivated by LFA.  For both methods, an
increase in the number of smoothing steps, roughly linearly with the spline
degree, is required in order to maintain robust convergence.  They can thus not
be considered totally robust and efficient in the strict meaning that we will
use in the present work.

A second approach towards robust and efficient multigrid was presented
in \cite{HTZ:2016}. Based on a robust inverse inequality and approximation
error estimate in a large subspace of maximally smooth spline spaces derived in
\cite{Takacs:Takacs:2015}, it was shown that mass matrices can be used
as robust smoothers in this large subspace. For the remaining, relatively
few degrees of freedom, a low-rank correction was constructed.
(These degrees of freedom are associated with the boundary of the domain
and cannot be captured by LFA, which assumes periodic boundary conditions.)
This approach resulted in a provably robust and efficient multigrid
method for two-dimensional problems with splines of maximum smoothness.
It was however not clear how to extend this approach efficiently to
three and higher dimensions.

The present work can be viewed as a continuation of \cite{HTZ:2016}.
Based on the theoretical results from \cite{Takacs:Takacs:2015}, we construct
a splitting of the tensor product spline space into a large, regular interior
part and several smaller spaces which capture boundary effects.  The
splitting is $L_2$-orthogonal and $H^1$-stable with respect to both
the mesh size and the spline degree.
This stability enables us to construct a multigrid smoother based on
an additive subspace correction approach, applying a different smoother
in each of the subspaces.
In the regular interior subspace, we use a mass smoother. In the other
subspaces, we construct smoothers which exploit the particular structure
of the subspaces while still permitting an efficient application through
a Kronecker product representation. In one small subspace associated with
the corners of the domain, we apply a direct solver.

Unlike the low-rank correction approach from \cite{HTZ:2016}, the subspace
correction approach generalizes easily to three-dimensional problems, and
indeed to problems of arbitrary space dimension.  We show that the method
converges robustly with respect to mesh size and spline degree, and that one
iteration is asymptotically not more expensive than an application of the
stiffness matrix.  The result is a quasi-optimal solution method for problems
of arbitrary space dimensions.

It appears that the stable splitting of the tensor product spline space
presented in Section~\ref{sec:splitting} is an interesting theoretical result
in its own right. It may have future applications to other aspects of IgA
beyond the one presented here.

The remainder of the paper is organized as follows.
In Section~\ref{sec:preliminaries}, we introduce the needed spline spaces and
present an isogeometric model problem. We also present an algorithmic multigrid
framework and an abstract convergence result which forms the basis of our later
analysis.
In Section~\ref{sec:splitting}, we derive the main new theoretical result used
in our construction: the $L_2$-orthogonal and $H^1$-stable splitting of the
spline space into a large, regular interior part and smaller spaces which
capture boundary effects.
In Section~\ref{sec:construction}, we use this space splitting to construct a
multigrid smoother based on the idea of additive subspace correction and show
that it results in a robust solver.
In Section~\ref{sec:realization}, we present details on the computational
realization of the proposed smoother and show that it permits an efficient
implementation in arbitrary space dimensions.
In Section~\ref{sec:experiments}, we present numerical experiments which
demonstrate the performance of the proposed method in practice.

%%%%%%%%%%%%%%%%%%%%%%%%%%%%%%%%%%%%%%%%%%%%%%%%%%%%%%%%%%%%%%%%%%%%%%%%%%%%%%%%

\section{Preliminaries}
\label{sec:preliminaries}

\subsection{Spline spaces and B-splines}
\label{sec:splines}

Consider a subdivision of the interval $(0,1)$ into $m \in \mathbb N$
intervals of length $h = 1/m$.
We introduce the spline space of degree $p \in \mathbb N$ with
maximum smoothness,
\[
    S := \{ u \in C^{p-1}(0,1) : u|_{( (j-1)h, jh )} \in \mathcal P^p
        \quad \forall j=1,\ldots,m \},
\]
where $C^{p-1}(0,1)$ is the space of all $p-1$ times continuously differentiable functions
on $(0,1)$ and $\mathcal P^p$ is the space of all polynomials of degree at most $p$.
We have
$
    n := \dim S = m + p
$.
As a basis for $S$, we use the normalized (i.e., satisfying a partition of
unity; cf.~\cite{DeBoor:Practical}) B-splines with an open knot vector.  In
higher dimensions $d > 1$, we introduce the space of tensor product splines
(cf.~\cite{DeBoor:Practical})
\[
    S^{d} := S \otimes \ldots \otimes S
\]
defined over $(0,1)^d$ with $\dim S^d = n^d$
and the corresponding tensor product B-spline basis.
For notational convenience, we assume that the same spline space $S$ is used
in each of the $d$ coordinate directions. Both our construction and our
analysis are however straightforward to generalize to the case where different
spline spaces are used in different coordinate directions.

\subsection{Isogeometric model problem}
\label{sec:model-problem}

Let $\Omega = (0,1)^d$ with $d \in \mathbb N$.
As a model problem, we consider a pure Neumann boundary value problem
for the PDE $-\Delta u + u = f$.
The variational formulation reads:
find $u \in H^1(\Omega)$ such that
\begin{equation}
    \label{eq:vf}
    a(u,v) = \sprod f v    \qquad   \forall v \in H^1(\Omega),
\end{equation}
where
\begin{equation}
    \label{eq:bilinear}
    a(u,v) = \int_\Omega (\nabla u \cdot \nabla v  + u v ) \,dx
    \qquad
    \forall u,v \in H^1(\Omega)
\end{equation}
and $f$ is a linear functional on $H^1(\Omega)$.
We will sometimes refer to the operator
$A: H^1(\Omega) \to H^1(\Omega)'$ given by $A v = a(v, \cdot)$,
where $H^1(\Omega)'$ denotes the continuous dual.
Note that $\norm{v}_A^2 = a(v,v) = \norm{v}_{H^1(\Omega)}^2$.

Discretizing \eqref{eq:vf} using tensor product splines, we
seek $u_h \in S^{d}$ such that
\begin{equation}
    \label{eq:discr-vf}
    a(u_h,v_h) = \sprod f {v_h}    \qquad   \forall v_h \in S^{d}.
\end{equation}
We are interested in robust and efficient iterative solvers for the
discrete problem \eqref{eq:discr-vf}. Here, by ``robust'' we mean
that the number of iterations to solve the problem should stay uniformly
bounded with respect to both the mesh size $h$ and the spline degree $p$,
and by ``efficient'' we mean that one iteration of the method should
not be asymptotically more expensive than computing the product of the stiffness
matrix with a vector. Combined, these properties allow us to solve \eqref{eq:discr-vf}
in quasi-optimal time.

In IgA, one introduces a bijective geometry map
from $\Omega$ to the actual domain of interest in order to be able to treat
more complicated computational domains.
Basis functions on the transformed domain are defined by composing the basis
functions on the reference domain with the inverse of the geometry map.
Furthermore, one is often interested in more general PDEs with varying and
possibly matrix-valued coefficients.
Discretizations for such more general problems can be preconditioned with a
solver for the model problem \eqref{eq:discr-vf}, and the resulting
condition number depends only on the geometry map and the coefficient
functions, but not on discretization parameters like the mesh size $h$ or the
spline degree $p$.  This principle has been widely used in the literature on
IgA solvers (see, e.g., \cite{Donatelli:2014b,HTZ:2016}) and formalized
in~\cite{Sangalli:2016}.
Therefore, a robust and efficient solver for the model problem
\eqref{eq:discr-vf} immediately yields robust and efficient solvers for a more
general class of problems with ``benign'' geometry maps and mildly varying
coefficients. This justifies the study of solvers for the model problem.

Three different refinement strategies for IgA discretizations were proposed in
\cite{Hughes:2005}: $h$-refinement (reducing the mesh size), $p$-refinement
(increasing the spline degree), and the so-called $k$-refinement.  The latter
is unique to IgA and maintains the maximum possible smoothness $C^{p-1}$ for
the spline space of degree $p$.  Already in \cite{Hughes:2005}, the favorable
performance of $k$-refinement was observed, and it appears to be the most
popular refinement strategy in the wider IgA literature.  This motivates the
study of solvers for spline spaces with maximum smoothness.

\subsection{A multigrid method framework}
\label{sec:mg}

Given a discretization space $V$ and a coarse space $V_c \subset V$,
we denote by $P: V_c \to V$ the canonical embedding.
Let $A: V \to V'$ denote the operator in a (discretized) equation
\[
    Au = f
\]
to be solved for $u \in V$.
The corresponding coarse-space operator is given by $A_c := P' A P$.
Furthermore, we assume that we are given a self-adjoint and positive definite
smoothing operator $L: V \to V'$.

Given a previous iterate $u^{(k)}$, we let $u^{(k,0)} := u^{(k)}$ and perform
$\nu \in \mathbb N$ \textit{smoothing steps} given by
\[
    u^{(k,j)} := u^{(k,j-1)} +  \tau L^{-1} (f - A u^{(k,j-1)}),
    \qquad
    j = 1, \ldots, \nu,
\]
where $\tau > 0$ is a damping parameter.
Then, we perform one \textit{coarse-grid correction step} given by
\[
    u^{(k+1)} := u^{(k,\nu)} + P A_c^{-1} P' (f - A u^{(k,\nu)}).
\]
Together, these updates describe one iteration $u^{(k)} \mapsto u^{(k+1)}$ of
a \textit{two-grid method}.
Given an entire sequence of nested spaces $V_0 \subset \ldots \subset V_L = V$, we
can replace the exact inversion of $A_c$ in the coarse-grid correction step
by one or two recursive applications of the two-grid method on the next coarser
level $V_{L-1}$, and so on until we reach the coarsest level $V_0$, where an
exact solver is used.
Using one or two recursive iteration steps results in the
\textit{V-cycle} or the \textit{W-cycle multigrid method}, respectively.

The following theorem is an abstract convergence result for the two-grid method
with the abovementioned smoother.
Its proof is given in \cite[Theorem~3]{HTZ:2016} and is based on
a variant of the standard multigrid theory as developed by Hackbusch~\cite{Hackbusch:1985}.
In~\cite[Theorem~4]{HTZ:2016},
it was shown that under the same assumptions also a W-cycle multigrid method converges.

\begin{theorem}[\cite{HTZ:2016}]
    \label{thm:twogrid}
    Assume that there are constants $C_A$ and $C_I$ such that the inverse inequality
    \begin{equation}\label{eq:sp}
        \norm{u}_{A}^2 \le C_I \norm{u}_{L}^2
        \qquad
        \forall u \in V
    \end{equation}
    and the approximation property for the $A$-orthogonal projector
    $T_c : V \to V_c$
    \begin{equation}\label{eq:ap}
        \norm{(I-T_c) u}_L^2 \le C_A \norm{u}_A^2
        \qquad
        \forall u \in V
    \end{equation}
    hold.
    Then the two-grid method converges for any choice of the damping parameter $\tau\in(0, C_I^{-1}]$ and
    any number of smoothing steps $\nu > \nu_0 := \tau^{-1}C_A$ with rate $q = \nu_0/\nu<1$.
\end{theorem}

In particular, if $C_A$ and $C_I$ do not depend on
the mesh size $h$ and the spline degree $p$, then the two-grid method converges
with a rate $q < 1$ which does not depend on $h$ and $p$.
In other words, the two-grid method is then robust.

In addition to properties \eqref{eq:sp} and \eqref{eq:ap}, care must be taken
that the smoother can be realized efficiently. In other words, it should be
possible to apply the inverse $L^{-1}$ with a computational cost which is
roughly comparable to that for applying $A$.

%%%%%%%%%%%%%%%%%%%%%%%%%%%%%%%%%%%%%%%%%%%%%%%%%%%%%%%%%%%%%%%%%%%%%%%%%%%%%%%%

\section{Stable splittings of spline spaces}
\label{sec:splitting}

Consider first the univariate case, $d=1$, with $\Omega=(0,1)$.
In \cite{Takacs:Takacs:2015}, the subspace
\[
    S_0 := \left\{ u \in S: u^{(2l+1)}(0) = u^{(2l+1)}(1) = 0
      \quad \forall l \in \mathbb{N}_0 \text{ with } 2l+1 < p
    \right\}
\]
of splines with vanishing odd derivatives of order less than $p$ at the boundaries
was introduced (denoted in~\cite{Takacs:Takacs:2015} by $\widetilde S_{p,h}(\Omega)$).
It is a large subspace of $S$ in the sense that
$
    \dim S_0 \ge \dim S - p.
$

The subspace $S_0$ has the very desirable property of satisfying
both a (first-order) approximation property and an inverse inequality, both
with constants which are independent of the spline degree $p$.
To formulate these results, let $Q_0: L_2(\Omega) \rightarrow S_0$ denote the
$L_2$-orthogonal projector into $S_0$, and let $\Pi_0: H^1(\Omega) \rightarrow S_0$
denote the projector into $S_0$ which is orthogonal with respect to the scalar product
\[
   (u,v)_{H^1_\circ(\Omega)} := (\nabla u,\nabla v)_{L_2(\Omega)}
    + \frac{1}{|\Omega|} \left(\int_\Omega u(x) dx\right)\left(\int_\Omega v(x) dx\right).
\]

We abbreviate the $L_2(\Omega)$-norm by $\norm{\cdot}_0$,
and the full $H^1(\Omega)$-norm and the seminorm
by $\norm{\cdot}_1$ and $|\cdot|_1$, respectively.
Furthermore, we write $c$ for a generic positive constant which does not depend
on the mesh size $h$ or the spline degree $p$.

\begin{theorem}[{\cite[Theorem~6.1]{Takacs:Takacs:2015}}]
    \label{thm:inverse}
    For any spline degree $p\in \mathbb{N}$,
    we have the inverse inequality
    \[
        |u|_{1} \le 2 \sqrt{3} h^{-1} \|u\|_{0}
        \qquad \forall u \in S_0.
    \]
\end{theorem}

\begin{theorem}[{\cite[Corollary~5.1]{Takacs:Takacs:2015}, \cite[Theorem~14]{HTZ:2016}}]
    \label{thm:approx}
    For any spline degree $p\in \mathbb{N}$ and any $u \in H^1(\Omega)$,
    we have the approximation error estimates
    \[
        \norm{(I - Q_0)   u}_{0} \le   \sqrt{2} h |u|_{1}
        \qquad\mbox{and}\qquad
        \norm{(I - \Pi_0) u}_{0} \le   \sqrt{2} h |u|_{1}.
    \]
\end{theorem}

Contrast these properties with the entire spline space $S$, which does satisfy
a robust approximation property, but whose inverse inequality deteriorates with
increasing degree $p$ (\cite{Takacs:Takacs:2015}).  On the other hand, a
smaller space of only ``interior'' splines, built by discarding the $p$
leftmost and $p$ rightmost B-splines, does satisfy a robust inverse inequality
but loses the approximation property.

We remark that the non-robustness of the inverse inequality in $S$
is the root cause of the spectral ``outliers'' commonly observed when
solving eigenvalue problems using IgA (cf.~\cite{Cottrell:2006}).
No such outliers appear in the space $S_0$.

\subsection{A stable splitting in one dimension}
\label{sec:splitting-1d}

Let $S_1 := S_0^{\bot_{L_2}}$ denote the $L_2$-orthogonal complement of $S_0$
in $S$.  Consider the splitting of $S$ into the direct sum
\[
    S = S_0 \oplus S_1
    \quad \longleftrightarrow \quad
    u = Q_0 u + (I - Q_0) u
\]
of $S_0$ and its complement, illustrated in Fig.~\ref{fig:splitting-1d}.
Due to orthogonality, we have
\begin{equation}\label{eq:l2:orth}
    \norm{u}_0^2 = \norm{Q_0 u}_0^2 + \norm{(I - Q_0) u}_0^2.
\end{equation}

\begin{figure}[htb]
    \centering
    \includegraphics[width=0.49\textwidth]{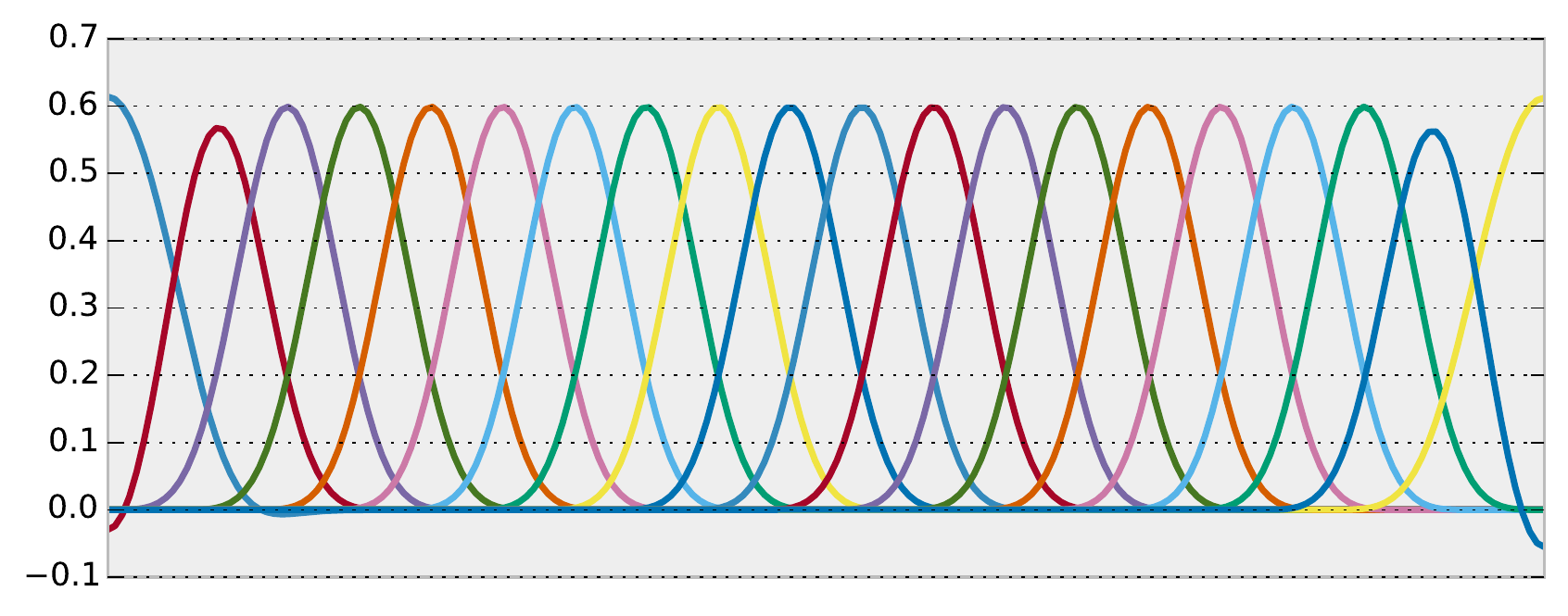}
    \includegraphics[width=0.49\textwidth]{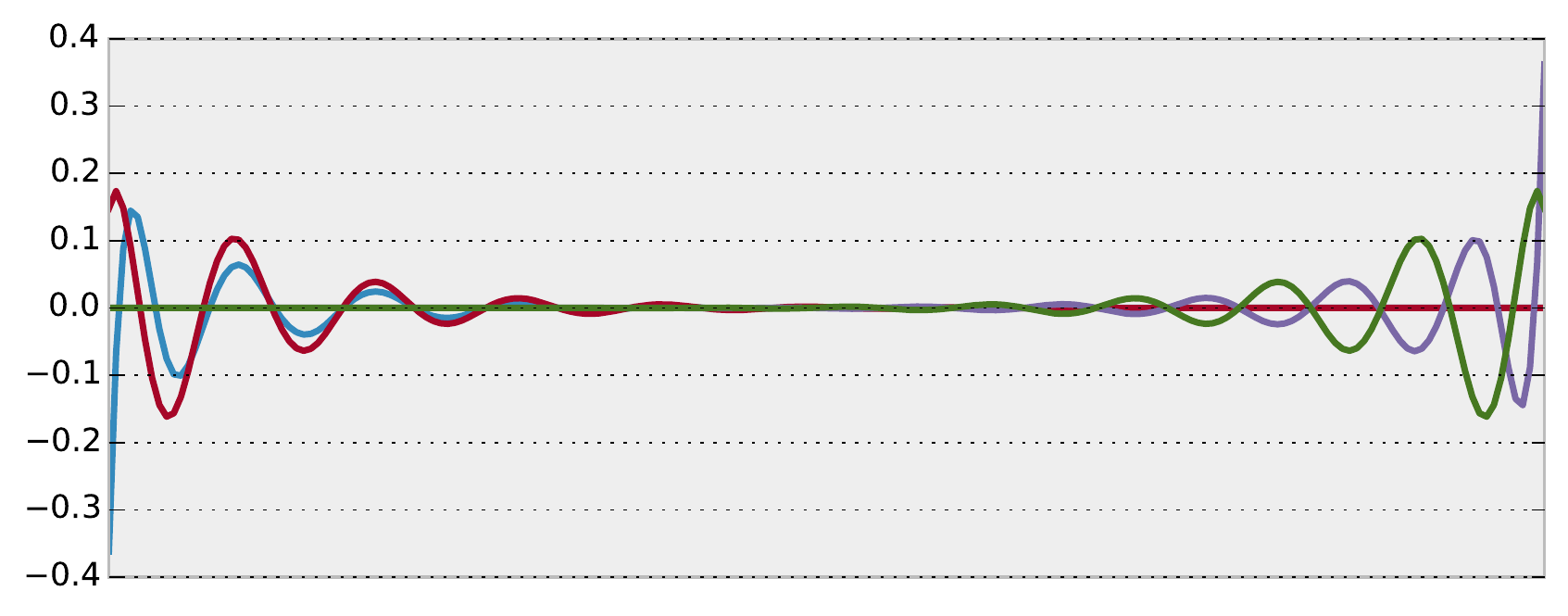}
    \caption{Bases for the space $S_0$ (left) and its orthogonal complement $S_1$ (right)
        for $p=4$, $h=1/20$. Here, $\dim S_0 = 20$ and $\dim S_1 = 4$.}
    \label{fig:splitting-1d}
\end{figure}

Crucially, we can prove that this splitting is stable also in the $H^1$-norm.
This is a direct result of the space $S_0$ satisfying both an approximation
property and an inverse inequality.

\begin{theorem}
    \label{thm:stable-1d}
    For any spline $u \in S$, we have
    \[
        c^{-1} |u|_1^2 \le
        |Q_0 u|_1^2 + |(I - Q_0) u|_1^2
        \le c |u|_1^2
    \]
    and the corresponding result for the full $H^1$-norm.
\end{theorem}
\begin{proof}
    The left inequality follows from the Cauchy-Schwarz inequality with $c=2$.
    For the right inequality, we observe that
    \[
        | Q_0 u |_1
        \le
        | \Pi_0 u |_1 + | (\Pi_0 - Q_0) u |_1
        \le
        | u |_1 + c h^{-1} \left( \norm{ (I - \Pi_0) u }_0 +  \norm{ (I - Q_0) u }_0 \right)
    \]
    because of the triangle inequality, the stability of the $H^1_\circ$-projector
    $\Pi_0$ in the $H^1$-seminorm and the robust inverse inequality in
    $S_0$ (Theorem~\ref{thm:inverse}).
    With the approximation error estimate (Theorem~\ref{thm:approx}) we obtain $H^1$-stability of the $L_2$-projector,
    \begin{equation}\label{eq:proof:thm:stable-1d}
        | Q_0 u |_1 \le c | u |_1.
    \end{equation}
    The desired result follows from~\eqref{eq:proof:thm:stable-1d} and
    \[
        |(I - Q_0) u|_1
        \le
        |u|_1 +
        |Q_0 u|_1
        \le
        (1+c) |u|_1.
    \]
    The result for the full $H^1$-norm follows by adding the identity \eqref{eq:l2:orth}.
\end{proof}

\subsection{A stable splitting in two dimensions}
\label{sec:splitting-2d}

The two-dimensional tensor product spline space is given by
$
    S^{2} = S \otimes S.
$
Since the tensor product distributes over direct sums, we obtain the splitting
\[
    S^{2}
    =
    (S_0 \otimes S_0) \oplus
    (S_0 \otimes S_1) \oplus
    (S_1 \otimes S_0) \oplus
    (S_1 \otimes S_1)
    =
    S_{00} \oplus S_{01} \oplus S_{10} \oplus S_{11}
\]
with the abbreviations
$
    S_{\alpha_1,\alpha_2} := S_{\alpha_1} \otimes S_{\alpha_2}
$
for
$
    \alpha_j \in \{0,1\}
$.
A visualization of this splitting is shown in Fig.~\ref{fig:splitting-2d}.
Note that the shaded regions do not correspond to the supports of the function
spaces; in fact, each of the subspaces has global support.  However, the shaded
regions roughly correspond to regions where the corresponding functions are
``largest'', and their areas roughly correspond to the space dimensions.  In
view of this, it makes sense to think of $S_{00}$ as an ``interior'' space, of
$S_{01}$ and $S_{10}$ as ``edge'' spaces, and of $S_{11}$ as a ``corner''
space.

\begin{figure}[htb]
    \centering
    \begin{minipage}[c]{0.2\textwidth} \centering
        \includegraphics[width=\textwidth]{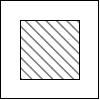} \\ $S_{00}$
    \end{minipage}
    \raisebox{1ex}{$\oplus$}
    \begin{minipage}[c]{0.2\textwidth} \centering
        \includegraphics[width=\textwidth]{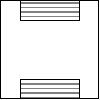} \\ $S_{01}$
    \end{minipage}
    \raisebox{1ex}{$\oplus$}
    \begin{minipage}[c]{0.2\textwidth} \centering
        \includegraphics[width=\textwidth]{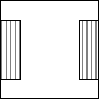} \\ $S_{10}$
    \end{minipage}
    \raisebox{1ex}{$\oplus$}
    \begin{minipage}[c]{0.2\textwidth} \centering
        \includegraphics[width=\textwidth]{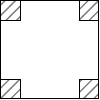} \\ $S_{11}$
    \end{minipage}
    \caption{Visualization of the splitting in 2D.}
    \label{fig:splitting-2d}
\end{figure}

Again, we can prove that the splitting is $H^1$-stable.
In the following, we let $M: S \to S'$, $K: S \to S'$
denote the operators in the univariate spline space associated with
the bilinear forms
\[
    \sprod{M u}{v} := \int_0^1 u(x) v(x) \, dx,
    \quad
    \sprod{K u}{v} := \int_0^1 u'(x) v'(x) \, dx
    \quad
    \forall u,v \in S,
\]
that is, the one-dimensional mass and stiffness operators, respectively.
For any $(\alpha_1,\alpha_2) \in \{0,1\}^2$,
we furthermore introduce the abbreviations
\begin{alignat*}{2}
    Q_1 &:= I - Q_0 : S \to S_1,
    \qquad
    & Q_{\alpha_1,\alpha_2} &:= Q_{\alpha_1} \otimes Q_{\alpha_2}: S^2 \to S_{\alpha_1,\alpha_2}
    \\
    K_{\alpha_j} &:= Q_{\alpha_j}' K Q_{\alpha_j}: S_{\alpha_j} \to S_{\alpha_j}',
    \qquad
    & M_{\alpha_j} &:= Q_{\alpha_j}' M Q_{\alpha_j}: S_{\alpha_j} \to S_{\alpha_j}'.
\end{alignat*}
As tensor products of $L_2(0,1)$-orthogonal projectors, the projectors
$Q_{\alpha_1,\alpha_2}$ are
$L_2(\Omega)$-orthogonal, as one easily verifies. Thus the splitting of $S^{2}$
given above is a direct sum of $L_2$-orthogonal subspaces, and we have
\begin{equation}
    \label{eq:ortho-2d}
    \norm{u}_0^2
    =
    \sum_{(\alpha_1,\alpha_2)}  \norm{ Q_{\alpha_1,\alpha_2} u }_0^2,
\end{equation}
where here and below sums over $(\alpha_1,\alpha_2)$ are taken to run over
the set $\{0,1\}^2$.

\begin{theorem}\label{thm:stable-2d}
    For any tensor product spline $u \in S^{2}$, we have
    \begin{equation*}
        c^{-1} |u|_1^2
        \le
        \sum_{(\alpha_1,\alpha_2)}  | Q_{\alpha_1,\alpha_2} u |_1^2
        \le
        c |u|_1^2,
    \end{equation*}
    and the corresponding result for the full $H^1$-norm.
\end{theorem}
\begin{proof}
    The left inequality follows by the Cauchy-Schwarz inequality.
    For the right one, fix $(\alpha_1,\alpha_2) \in \{0,1\}^2$.
    The $H^1$-seminorm can be written using tensor products of
    one-dimensional operators as
    \begin{equation}
        \label{eq:eqn-1}
        | Q_{\alpha_1,\alpha_2} u |_1^2 =
        | Q_{\alpha_1,\alpha_2} u |_{K \otimes M}^2 +
        | Q_{\alpha_1,\alpha_2} u |_{M \otimes K}^2.
    \end{equation}
    The first term can be rewritten, using the definitions and basic identities
    for tensor products of operators, as
    \[
        | Q_{\alpha_1,\alpha_2} u |_{K \otimes M}^2
        = \sprod {Q_{\alpha_1,\alpha_2}' (K \otimes M) Q_{\alpha_1,\alpha_2} u} {u} %\\
        = \sprod {(K_{\alpha_1} \otimes M_{\alpha_2}) u} {u}.
    \]
    Due to orthogonality and Theorem~\ref{thm:stable-1d},
    we have
        $M_0 + M_1 = M$ and
        %\qquad
        $K_0 + K_1 \le c K$,
    where all summands are positive semidefinite operators.
    This implies that we can estimate, in the spectral sense,
    $K_{\alpha_1} \le cK$ and $M_{\alpha_2} \le M$,
    and we obtain
    \[
        | Q_{\alpha_1,\alpha_2} u |_{K \otimes M}^2
        \le c | u |_{K \otimes M}^2.
    \]
    Treating the second term in \eqref{eq:eqn-1} analogously, we obtain
    \[
        | Q_{\alpha_1,\alpha_2} u |_1^2 \le c (| u |_{K \otimes M}^2 + | u |_{M \otimes K}^2)
        = c | u |_1^2.
    \]
    The right inequality now follows by summing up over all $(\alpha_1,\alpha_2)$.
    The result for the full $H^1$-norm follows by adding the identity \eqref{eq:ortho-2d}.
\end{proof}

\subsection{Stable splitting in arbitrary dimensions}
\label{sec:splitting-nd}

For any $d \in \mathbb N$, we define multiindices $\alpha \in \{0,1\}^d$
and generalize the notations from Section~\ref{sec:splitting-2d} in the straightforward
way to higher dimensions.
We obtain the splitting into the direct sum of $2^d$ subspaces
\[
    S^d = \bigoplus_{\alpha} S_\alpha,
    \qquad \text{where} \qquad
    S_\alpha = S_{\alpha_1} \otimes \ldots \otimes S_{\alpha_d}.
\]
The $L_2$-orthogonal projectors into the subspaces are given by
\[
    Q_{\alpha} = Q_{\alpha_1} \otimes \ldots \otimes Q_{\alpha_d}: S^{d} \to S_\alpha.
\]
As in the two-dimensional case, we can prove that this splitting is
$H^1$-stable.

\begin{theorem}\label{thm:stable-nd}
    For any $d$-dimensional tensor product spline $u \in S^{d}$, we have
    \[
        c^{-1} |u|_1^2
        \le
        \sum_{\alpha=(0,\ldots,0)}^{(1,\ldots,1)}  | Q_{\alpha} u |_1^2
        \le
        c |u|_1^2
    \]
    and the corresponding result for the full $H^1$-norm.
\end{theorem}
\begin{proof}
    Completely analogous to Theorem~\ref{thm:stable-2d}.
\end{proof}

%%%%%%%%%%%%%%%%%%%%%%%%%%%%%%%%%%%%%%%%%%%%%%%%%%%%%%%%%%%%%%%%%%%%%%%%%%%%%%%%

\section{Construction of a robust multigrid method}
\label{sec:construction}

Recall that $S$ was a univariate spline space of degree $p$ and mesh size $h$.
Let $S_c \subset S$ be the analogous coarse spline space with uniform mesh size
$2h$.  For the construction of our two-grid method in $d$ dimensions in
accordance with the framework introduced in Section~\ref{sec:mg}, we let
\[
    V := S^d,
    \qquad
    V_c := (S_c)^d \subset V.
\]
The prolongation $P : V_c \to V$ is the canonical embedding of the coarse
tensor product spline space in the fine one.
It can be represented as the $d$-fold tensor product of prolongations
for the univariate spline spaces, $I: S_c \to S$.

The following result states that a robust approximation error estimate holds
for the Galerkin projector to the coarse spline space.  It was proved for $d=1$
and $d=2$ in \cite{HTZ:2016}. We extend the proof to arbitrary dimensions in
the Appendix.

\begin{lemma}\label{lemma:approx:dd}
    The $A$-orthogonal projector $T_c : S^d \to (S_c)^d$ satisfies the
    approximation error estimate
    \begin{equation*}
        \| (I-T_c) u \|_{L_2(\Omega)} \le c h \|u\|_{A}
        \qquad \forall u\in S^d
    \end{equation*}
    with a constant $c$ which is independent of $h$ and $p$ (but may depend on $d$).
\end{lemma}

In the following subsections, we construct a smoother for the two-grid
method on these nested spline spaces which leads to a robust and efficient
iterative method.

\subsection{A multigrid smoother based on subspace correction}

In each of the $2^d$ subspaces $S_\alpha \subset S^{d}$, $\alpha \in \{0,1\}^d$,
defined in Section~\ref{sec:splitting-nd},
we prescribe a local, symmetric and positive definite smoothing operator
$L_\alpha: S_\alpha \to S_\alpha'$.  The overall smoothing operator is then
given by the additive subspace operator
\begin{equation}
    \label{eq:smoother}
    L := \sum_\alpha Q_\alpha' L_\alpha Q_\alpha : S^{d} \to S^{d'},
\end{equation}
from $S^d$ to its dual $S^{d'}$,
and its inverse has the form
\[
    L^{-1} = \sum_\alpha L_\alpha^{-1} Q_\alpha' : S^{d'} \to S^{d}.
\]
The assumptions of Theorem~\ref{thm:twogrid} for $L$, and thus the convergence
of the two-grid method with such a smoother, can be guaranteed under simple
assumptions on the subspace operators $L_\alpha$, as the following two lemmas
show.
The stability of the space splitting is crucial to both proofs.
Although we do not explicitly use any results from the literature on subspace
correction methods, we rely heavily on the ideas developed therein;
cf., e.g., \cite{Xu:1992,GriebelOswald:1995a}.

\begin{lemma}
    \label{lem:smoothing-prop}
    Assume that for every $\alpha \in \{0,1\}^d$, we have
    \begin{equation}\label{eq:lem:smoothing-prop}
        \sprod {A v_\alpha}{v_\alpha} \le c \sprod {L_\alpha v_\alpha}{v_\alpha}
        \qquad
        \forall v_\alpha \in S_\alpha.
    \end{equation}
    Then the subspace correction smoother satisfies
    \[
        \sprod {A v}{v} \le c \sprod{L v}{v}
        \qquad \forall v \in S^{d}.
    \]
\end{lemma}
\begin{proof}
    Due to Theorem~\ref{thm:stable-nd} and~\eqref{eq:lem:smoothing-prop}, we have
    \[
        \sprod {A v}{v}
        \le
        c \sum_\alpha \sprod {A Q_\alpha v} {Q_\alpha v}
        \le
        c \sum_\alpha \sprod {L_\alpha Q_\alpha v} {Q_\alpha v}
        =
        c \sprod{L v}{v}.
    \]
\end{proof}

\begin{lemma}
    \label{lem:approx-prop}
    Assume that for every $\alpha \in \{0,1\}^d$, we have
    \begin{equation}\label{eq:lem:approx-prop}
        \sprod {L_\alpha v_\alpha}{v_\alpha} \le c \sprod {(A + h^{-2} M^d) v_\alpha}{v_\alpha}
        \qquad
        \forall v_\alpha \in S_\alpha,
    \end{equation}
    where $M^d: S^d \to S^{d'}$ is the mass operator in the tensor product
    spline space.
    Then the subspace correction smoother satisfies
    \[
        \norm{ (I - T_c) v }_L \le c \norm{v}_A
        \qquad \forall v \in S^{d}.
    \]
\end{lemma}
\begin{proof}
    From~\eqref{eq:lem:approx-prop}, Theorem~\ref{thm:stable-nd} and $L_2$-orthogonality, we obtain
    \[
        \sprod {L v}{v}
        \le
        c \sum_\alpha \sprod {(A + h^{-2} M^d) Q_\alpha v} {Q_\alpha v}
        \le
        c \sprod{(A + h^{-2} M^d) v}{v}.
    \]
    Thus, it follows
    \[
        \norm{ (I - T_c) v }_L^2
        \le c \norm{ (I - T_c) v }_A^2 + c h^{-2} \norm{ (I - T_c) v }_{M^d}^2
        \le c \norm{  v }_A^2,
    \]
    where we used the stability of the coarse-grid projector and the
    coarse-grid approximation property Lemma~\ref{lemma:approx:dd}.
\end{proof}

\subsection{Choice of the local smoothing operators}
\label{sec:smoother}

We now construct suitable local operators $L_\alpha$ which satisfy the
assumptions of Lemma~\ref{lem:smoothing-prop} and Lemma~\ref{lem:approx-prop}.
In the two-dimensional case, the operator associated with the bilinear form
\eqref{eq:bilinear} admits the representation
\[
    A = K \otimes M + M \otimes K + M \otimes M
\]
in terms of the stiffness and mass operators for the univariate case.
Restricting $A$ to a subspace $S_\alpha = S_{\alpha_1,\alpha_2}$, we
obtain
\[
    A_\alpha := Q_\alpha' A Q_\alpha
    =
    K_{\alpha_1} \otimes M_{\alpha_2} +
    M_{\alpha_1} \otimes K_{\alpha_2} +
    M_{\alpha_1} \otimes M_{\alpha_2}.
\]
The inverse inequality in $S_0$ (Theorem~\ref{thm:inverse})
allows us to estimate
\[
    K_0 \le \sigma M_0,
\]
where $\sigma = 12 h^{-2}$.
We obtain subspace smoothers $L_\alpha$ by replacing $K_0$ by $\sigma M_0$,
\begin{alignat*}{2}
    A_{00} & \le (1+2 \sigma) M_0 \otimes M_0          && =: L_{00},   \\
    A_{01} & \le M_0 \otimes ((1+\sigma) M_1 + K_1)  && =: L_{01},   \\
    A_{10} & \le ((1+\sigma) M_1 + K_1) \otimes M_0  && =: L_{10},   \\
    A_{11} & = M_1 \otimes M_1 + K_1 \otimes M_1 + M_1 \otimes K_1  && =: L_{11},%Q_{11}' A Q_{11}
\end{alignat*}
where \eqref{eq:lem:smoothing-prop}, the assumption of
Lemma~\ref{lem:smoothing-prop}, holds by construction.
It is easy to see that each $L_\alpha$ can be spectrally bounded
from above by a constant times the matrix $Q_\alpha' (A + h^{-2} M \otimes M) Q_\alpha$,
which proves the assumption~\eqref{eq:lem:approx-prop} of Lemma~\ref{lem:approx-prop}. Using
the statements of these two lemmas, Theorem~\ref{thm:twogrid} implies the two-grid convergence.

The same approach generalizes to higher dimensions, and we illustrate this
in the three-dimensional setting.
Here, we have
\[
    A = K \otimes M \otimes M + M \otimes K \otimes M + M \otimes M \otimes K + M \otimes M \otimes M.
\]
Again, we define $A_{\alpha}$ as above and obtain the operators $L_\alpha$ by replacing $K_0$ by $\sigma M_0$,
\begin{alignat*}{2}
    A_{000} & \le (1+3 \sigma) M_0 \otimes M_0 \otimes M_0          && =: L_{000},   \\
    A_{001} & \le M_0 \otimes M_0 \otimes ((1+2 \sigma) M_1 + K_1)  && =: L_{001},   \\
    A_{010} & \le M_0 \otimes ((1+2 \sigma) M_1 + K_1) \otimes M_0  && =: L_{010},   \\
    A_{100} & \le ((1+2 \sigma) M_1 + K_1) \otimes M_0 \otimes M_0  && =: L_{100},   \\
    A_{011} & \le M_0 \otimes ((1+\sigma) M_1 \otimes M_1 + K_1 \otimes M_1 + M_1 \otimes K_1) && =: L_{011},   \\
    A_{110} & \le ((1+\sigma) M_1 \otimes M_1 + K_1 \otimes M_1 + M_1 \otimes K_1) \otimes M_0 && =: L_{110},   \\
    A_{101} & \le K_1 \otimes M_0 \otimes M_1 + (1+\sigma) M_1 \otimes M_0 \otimes M_1 + M_1 \otimes M_0 \otimes K_1 && =: L_{101}, \\
    A_{111} & = M_1 \otimes M_1 \otimes M_1 + K_1 \otimes M_1 \otimes M_1 + M_1 \otimes K_1 \otimes M_1 + M_1 \otimes M_1 \otimes K_1 \hspace{-.1cm}    && =: L_{111}. %Q_{111}' A Q_{111}
\end{alignat*}
We point out that, whereas $L_{011}$ and $L_{110}$ permit a tensor product
factorization, the operator $L_{101}$ cannot directly be factorized due to the
ordering of the involved spaces. However, the tensor product space $S_{101}$
is isomorphic to $S_{011}$ by a simple swapping of the order of the involved
tensor products.  We exploit this in Section~\ref{sec:matrix-inversion} below
by a simple renumbering of the degrees of freedom in order to obtain an
efficient method for inverting $L_{101}$.

It is clear that the rule of replacing $K_0$ by $\sigma M_0$ in each
operator $A_\alpha$ to obtain $L_\alpha$ extends directly
to arbitrary dimension $d$.
By the same arguments as above, we see that the
resulting subspace correction smoother satisfies the assumptions of
Lemma~\ref{lem:smoothing-prop} and Lemma~\ref{lem:approx-prop}.
Thus Theorem~\ref{thm:twogrid} shows that the resulting two-grid method
converges robustly with respect to $h$ and $p$.
We summarize this in the following theorem.

\begin{theorem}
    For any $d \in \mathbb N$,
    there exist choices for $\tau$ and $\nu$, independent of $h$ and $p$, such
    that the two-grid method in $S^{d}$ with the smoother induced by the
    subspace operators $L_\alpha$ as constructed above converges with a rate $q<1$
    which does not depend on the grid size $h$ or the spline degree $p$.
\end{theorem}

The robust convergence of the W-cycle multigrid method follows using standard
arguments, cf.~\cite{Hackbusch:1985,HTZ:2016}.

%%%%%%%%%%%%%%%%%%%%%%%%%%%%%%%%%%%%%%%%%%%%%%%%%%%%%%%%%%%%%%%%%%%%%%%%%%%%%%%%

\section{Computational realization}
\label{sec:realization}

In Section~\ref{sec:construction}, we have proposed a smoother
and shown that it leads to a robust two-grid method.
In this section, we provide details on the realization of the
method and show that it permits an efficient implementation.

\subsection{Computation of a basis for $S_0$ and $S_1$}
\label{sec:compute-Stilde}

In order to be able to work with the space $S_0$ and its orthogonal
complement, we require bases for them. The aim of this subsection is to provide
an algorithm for computing such bases as linear combinations of B-splines.

Recall that the univariate spline space $S$ with $m$ knot spans of width
$h=1/m$, degree $p$ and maximum smoothness $C^{p-1}$ has dimension $n = m + p$.
Let
\[
    \mathcal B := \{ \varphi_1, \ldots, \varphi_n \}
\]
denote the normalized (i.e., satisfying a partition of unity,
cf.~\cite{DeBoor:Practical}) B-spline basis of $S$.
We have $\operatorname{supp} \varphi_j = [(j - p - 1) h, j h] \cap [0,1]$.
All interior B-splines
\[
    \mathcal B^I := \{ \varphi_{p+1}, \ldots, \varphi_{n-p} \}
\]
vanish with all their derivatives up to the $p-1$st at the boundaries of
the interval $[0,1]$ and therefore lie in $S_0$.
(Here and in the following we assume that $p+1 \le m$ such that $\mathcal B^I$
is nonempty.)

It remains to find linear combinations of the first and last $p$ B-splines which
complete $\mathcal B^I$ to a basis of $S_0$.
Recall that $u \in S$ lies in $S_0$ iff
\[
    u^{(2l+1)}(0) = u^{(2l+1)}(1) = 0
  \quad \forall l \in \mathbb{N}_0 \text{ with } 2l+1 < p.
\]
Consider first the left boundary.  We need to satisfy
$k := \lfloor p/2 \rfloor$ conditions on the derivatives of the splines.
Let
\[
    \tilde D = \left( h^{2i-1} \varphi_j^{(2i-1)}(0) \right)_{i=1,\dots,k, \;  j = 1,\dots,p}
    \in \mathbb R^{k \times p}
\]
denote the matrix of the relevant B-spline derivatives at $0$, scaled with a
suitable power of $h$ in order to avoid numerical instabilities.  We pad
$\tilde D$ with $p-k$ zero rows to obtain a square matrix $D \in \mathbb R^{p \times p}$.
Computing the singular value decomposition (SVD), we obtain
\[
    D = U \Sigma V^\top
\]
with $U,V \in \mathbb R^{p \times p}$ orthogonal and
$\Sigma \in \mathbb R^{p \times p}$ being the diagonal matrix of singular values
in descending order.
By construction, $\Sigma$ contains $k$ nonzero and $p - k$ zero singular values.
Therefore, the rightmost $p-k$ columns of $V$ span the kernel of $D$,
and the linear combinations
\[
    \mathcal B_0^L := \left\{
        \sum_{i=1}^p V_{i,j} \varphi_i : j=p-k+1, \ldots, p
    \right\}
\]
lie in $S_0$.
By the analogous procedure at the right boundary, we compute a set
$\mathcal B_0^R$ of $p-k$ linear combinations of the last $p$ B-splines.
Then, the functions in the set
\[
    \mathcal B_0 := \mathcal B_0^L \cup B^I \cup \mathcal B_0^R
\]
are by construction linearly independent and lie in $S_0$.
Since $n_0 := |\mathcal B_0| = n - 2k = \dim{S_0}$, we have
\[
    \operatorname{span} {\mathcal B_0} = S_0.
\]

In practice, we collect the coefficients in a sparse block diagonal matrix
\[
    P_0 =
    \begin{bmatrix}
        V^L[:, p-k+1:p]  &          &  \\
                         & I_{n-2p} &  \\
                         &          & V^R[:,p-k+1:p]
    \end{bmatrix}
    \in \mathbb R^{n \times n_0},
\]
where $V^L[:, p-k+1:p] \in \mathbb R^{p \times (p-k)}$ denotes the last $p-k$
columns of the matrix $V$ computed for the left boundary, analogously $V^R$
that for the right boundary, and $I_d$ is the $d\times d$ identity matrix.
Then clearly, splines in $S_0$ can be uniquely represented in terms of the
B-spline basis as
\[
    u \in S_0
    \quad \Longleftrightarrow \quad
    \exists \underline u \in \mathbb R^{n_0}:
    u = \sum_{j=1}^n  (P_0 \underline u)_j \varphi_j.
\]

Since the SVD produces an orthonormal basis, collecting the remaining columns
of $V^L$ and $V^R$ in a second sparse block matrix
\[
    P_\bot =
    \begin{bmatrix}
        V^L[:, 1:k]    &  0 \\
             0         &  0 \\
             0         & V^R[:,1:k]
    \end{bmatrix}
    \in \mathbb R^{n \times 2k}
\]
satisfies $P_0^\top P_\bot = 0$.
In fact, the columns of the concatenation
$\begin{bmatrix}
    P_0 & P_\bot
\end{bmatrix}$
form an orthonormal basis of $\mathbb R^n$.
Let
\[
    P_1 := \underline M^{-1} P_\bot  \in \mathbb R^{n \times 2k},
\]
where $\underline M$ denotes the $\mathcal B$-mass matrix.
Note that $P_1$ is no longer sparse.
Furthermore, let
$\underline u \in \mathbb R^{n_0}$ and
$\underline v \in \mathbb R^{2k}$ with associated splines
\[
    u = \sum_{j=1}^n  (P_0 \underline u)_j \varphi_j,
    \qquad
    v = \sum_{j=1}^n  (P_1 \underline v)_j \varphi_j.
\]
By construction, $u \in S_0$.
We have
\[
    \sprod{u}{v}_{L_2(\Omega)} =
    \sprod{\underline M P_0 \underline u}{\underline M^{-1} P_\bot \underline v} =
    \underline u^\top P_0^\top P_\bot \underline v = 0.
\]
Since this holds for all $u \in S_0$,
$v$ lies in the $L_2$-orthogonal complement of $S_0$.
All in all, we have constructed basis representations or ``prolongation matrices''
\[
    P_0 \in \mathbb R^{n \times (n-2k)},
    \qquad
    P_1 = \underline M^{-1} P_\bot \in \mathbb R^{n \times 2k}
\]
for $S_0$ and its $L_2$-orthogonal complement $S_1$, respectively.

For $d > 1$, we let $\alpha \in \{0,1\}^d$ and introduce the Kronecker products
\[
    P_{\alpha} := P_{\alpha_1} \otimes \ldots \otimes P_{\alpha_d}
    \in \mathbb R^{n^d \times n_\alpha},
\]
where $n_\alpha = \dim{S_\alpha}$,
which represent bases for the spaces $S_\alpha$ in terms of the coefficients of
linear combinations of the tensor product B-spline basis $\mathcal B^{\otimes d}$.

\subsection{Implementation of the subspace correction smoother}
\label{sec:smoother-impl}

For any $\alpha \in \{0,1\}^d$, the matrices $P_\alpha$ as defined in
Section~\ref{sec:compute-Stilde} describe a basis for $S_\alpha$.
Let $\underline L_\alpha \in \mathbb R^{n_\alpha \times n_\alpha}$ be the
(symmetric and positive definite) matrix representation
of $L_\alpha: S_\alpha \to S_\alpha'$ (as defined in Section~\ref{sec:smoother})
with respect to that basis.
Then the matrix representation of
\[
    L^{-1} = \sum_\alpha L_\alpha^{-1} Q_\alpha'
    = \sum_\alpha I_{S_\alpha \to S^{d}} L_\alpha^{-1} I_{S^{d'} \to S_\alpha'} Q_\alpha' I_{S^{d'} \to S_\alpha'}
\]
is given by
\begin{equation}
    \label{eq:smoother-matrix}
    \underline L^{-1} = \sum_\alpha P_\alpha \underline L_\alpha^{-1} P_\alpha^\top \underline M P_\alpha \underline M_\alpha^{-1} P_\alpha^\top
    = \sum_\alpha P_\alpha \underline L_\alpha^{-1} P_\alpha^\top,
\end{equation}
where we used that the matrix representation of the embedding $I_{S_\alpha \to S^{d}}$
is $P_\alpha$ and the matrix representation of the $L_2$-projector $Q_\alpha$ is
\[
    \underline M_\alpha^{-1} P_\alpha^\top \underline M,
    \quad \text{where} \quad
    \underline M_\alpha = P_\alpha^\top \underline M P_\alpha.
\]
Hence \eqref{eq:smoother-matrix} can be used to implement the subspace correction
smoother using only the prolongation matrices $P_\alpha$ and a fast method for
applying $\underline L_\alpha^{-1}$.
It is never necessary to explicitly apply the $L_2$-projectors $Q_\alpha$.
Furthermore, due to the use of additive subspace correction, the residual needs
to be computed only once, and the individual subspace corrections may be done
in parallel.

\subsection{Inversion of the subspace operators}
\label{sec:matrix-inversion}

The final required algorithmic component is a fast method for applying the
inverse of the local smoothing matrices $\underline L_\alpha \in \mathbb
R^{n_\alpha \times n_\alpha}$.  We illustrate this in the three-dimensional
setting as described in Section~\ref{sec:smoother}, but the principles are
the same regardless of dimension.
A detailed discussion of the computational costs for arbitrary dimension
is given in Section~\ref{sec:complexity}.

\textbf{Interior space and face spaces.}
The interior space $S_{000}$ and the face spaces $S_{001}, S_{010}, S_{100}$
contain the complement space $S_1$ as a factor space at most once,
and thus the matrices associated with their smoothing operators can be represented
as Kronecker products of three one-dimensional discretization matrices, e.g.,
\[
    \underline L_{000} = (1+3 \sigma) \underline M_0 \otimes \underline M_0 \otimes \underline M_0,
    \quad
    \underline L_{001} = \underline M_0 \otimes \underline M_0 \otimes ((1+2 \sigma) \underline M_1 + \underline K_1).
\]
Here the symmetric matrices
$\underline M_\beta, \underline K_\beta \in \mathbb R^{\dim S_\beta \times \dim S_\beta}$,
$\beta\in\{0,1\}$,
are the matrix representations of $M_\beta$ and $K_\beta$, respectively, with
respect to the bases described by $P_\beta$ as computed in
Section~\ref{sec:compute-Stilde} above.
For $\beta=0$,
$\underline M_\beta$ and $\underline K_\beta$
have dimension $\mathcal O(n)$ and bandwidth $\mathcal O(p)$,
whereas for $\beta=1$ they have dimension $\mathcal O(p)$ and are dense.

Since the Kronecker product can be inverted componentwise, we obtain, e.g.,
\[
    \underline L_{001}^{-1} = \underline M_0^{-1} \otimes \underline M_0^{-1} \otimes ((1+2 \sigma) \underline M_1 + \underline K_1)^{-1}.
\]
Instead of computing this (dense) inverse explicitly, we employ the algorithm
described by de Boor \cite{DeBoor:1979} for computing the application of a
Kronecker product of matrices to a vector, given only routines for applying the
individual Kronecker factors. For the latter, we use Cholesky factorization.

\textbf{Edge spaces.}
The spaces $S_{011}, S_{110}, S_{101}$ contain the complement space $S_1$
as a factor twice. In $S_{011}$, the matrix to be inverted has the form
\[
    \underline L_{011} = \underline M_0 \otimes ((1+\sigma) \underline M_1 \otimes \underline M_1 + \underline K_1 \otimes \underline M_1 + \underline M_1 \otimes \underline K_1).
\]
It again has Kronecker product structure and can be inverted using the
algorithm described in the previous case.
The same holds for $S_{110}$.

In the case of the space $S_{101}$, the associated matrix
\[
    \underline L_{101} = \underline K_1 \otimes \underline M_0 \otimes \underline M_1 + (1+\sigma) \underline M_1 \otimes \underline M_0 \otimes \underline M_1 + \underline M_1 \otimes \underline M_0 \otimes \underline K_1
\]
does not permit a Kronecker product factorization due to the order of the
involved spaces. However, by a simple renumbering of the degrees of freedom,
$S_{101}$ can be identified with $S_{011}$, and then $\underline L_{011}^{-1}$
can be applied as above.

Alternatively, the matrix $\underline L_{101}$
could be directly computed and inverted in its entirety using Cholesky
factorization. This would exceed asymptotically (for $p\rightarrow \infty$)
the computational costs derived in the following subsection, however this
slowdown appears to be negligible in practice.
For $d>3$, this shortcut seems no longer viable.

\textbf{Corner space.}
The space $S_{111}$ is the tensor product of the three complement spaces and
has dimension $\dim(S_1)^3 \le p^3$. The associated matrix
\[
    \underline L_{111} =   M_1 \otimes M_1 \otimes M_1 + K_1 \otimes M_1 \otimes M_1 + M_1 \otimes K_1 \otimes M_1 + M_1 \otimes M_1 \otimes K_1%P_{111}^\top \underline A P_{111}
\]
is dense and is inverted by means of its Cholesky factorization.

\subsection{Computational costs}
\label{sec:complexity}

We now study the computational complexity for applying the subspace correction
smoother in the general $d$-dimensional setting.
In our analysis, we ignore multiplicative constants which depend only on $d$.
Repeatedly, we make use of the fact that the Cholesky factorization of a
symmetric matrix of dimension $N$ and bandwidth $q$ can be computed in
$\mathcal O(Nq^2)$ operations, and its inverse can then be applied in $\mathcal O(Nq)$ operations.
If the matrix is not banded but dense, the factorization and inversion require
$\mathcal O(N^3)$ and $\mathcal O(N^2)$ operations, respectively (cf.~\cite{golub2012matrix}).

By the renumbering of degrees of freedom described in Section~\ref{sec:matrix-inversion},
we can always rearrange the factor spaces such that we only need to
consider spaces of the form
\[
    \underbrace{S_0 \otimes \ldots \otimes S_0}_{k\text{ times}} \otimes
    \underbrace{S_1 \otimes \ldots \otimes S_1}_{d-k\text{ times}}.
\]
The smoothing matrices to be inverted, constructed as in Section~\ref{sec:matrix-inversion},
have the form
\[
    \underline L_{\{k,d-k\}} :=
      \underbrace{\underline M_0 \otimes \ldots \otimes \underline M_0}_{k\text{ times}}
      \otimes \underline X_{d-k},
\]
where $\underline X_{j} \in \mathbb R^{(\dim S_1)^{j} \times (\dim S_1)^{j} }$ is
a dense, symmetric matrix.
Recall that $\dim S_1 \le p$.

\textbf{Setup costs.}
The computation of the basis for $S_0$ and its $L_2$-orthogonal
complement as described in Section~\ref{sec:compute-Stilde} requires computing
the SVD of two matrices of dimension $\mathcal O(p)$
as well as $\mathcal O(p)$ applications of the inverse of $\underline M$,
which has dimension $n=m+p$ and bandwidth $\mathcal O(p)$, where $m$ is
the number of subintervals.
The costs for this step are thus $\mathcal O(p^3 + n p^2)=\mathcal O(p^3 + m p^2)$.

The one-dimensional mass matrix in $S_0$, $\underline M_0$,
has dimension $\mathcal O(m)$ and bandwidth $\mathcal O(p)$
and thus requires $\mathcal O(m p^2)$ operations to factorize.

The matrices $\underline X_j$, $j=1,\ldots,d$, are dense and therefore
require $\mathcal O(p^{3j})$ operations to factorize.

The overall setup costs are therefore
$
    \mathcal O(m p^2 + p^{3d})
$.

\textbf{Application costs.}
After factorization, the cost for applying the inverse
$\underline M_0^{-1}$ is $\mathcal O(mp)$,
and for $\underline X_j^{-1}$, it is $\mathcal O(p^{2j})$.
To apply $\underline L_{\{k,d-k\}}^{-1}$ using the Kronecker product algorithm
from \cite{DeBoor:1979}, we need to perform $m^{d-1}$ applications of each
of the $k$ factors $\underline M_0^{-1}$ and $m^k$ applications of $\underline X_j^{-1}$.
Thus, the cost is $ \mathcal O(k m^d p + m^k p^{2(d-k)}) $.

The inverse of $\underline L_{\{k,d-k\}}$ needs to be applied $\binom d k$ times
since that is the number of multiindices $\alpha \in \{0,1\}^d$ which permute
to $(0,\ldots,0,1,\ldots,1)$ with exactly $k$ leading zeros.
The binomial coefficient satisfies $\binom d k = \mathcal O(2^d / \sqrt{d})$
and in particular can be bounded from above by a constant which depends only on $d$.
The overall cost for one application of the subspace correction smoother is then
\[
    \sum_{k=0}^d \binom d k \mathcal O(k m^d p + m^k p^{2(d-k)})
    = \mathcal O\left(m^d p + \max_{k=0,\dots,d} m^k p^{2(d-k)} \right)
    = \mathcal O (m^d p + p^{2d}).
\]

\textbf{Overall costs.}
For $d \ge 2$, we have $m p^2\le m^2 + p^4 \le m^d + p^{2d}$. Therefore, the overall
costs for setting up and applying the smoother are bounded by
\[
    \mathcal O (m^d p + p^{3d}).
\]
Assuming $p^2 \lesssim m$, the overall costs are asymptotically not more expensive than one
application of the stiffness matrix, which has complexity $\mathcal O(n^d p^d)=\mathcal O(m^d p^d+p^{2d})$.

\newcommand{\mm}{\mathfrak{m}}
In a multigrid setting,
assuming $\mathcal O (\log \mm )$ levels with $m= \mm,  \tfrac \mm2, \tfrac \mm4, \tfrac \mm8 \ldots$
intervals per dimension, one obtains for $d\ge 2$ by summing up the overall costs of
\[
    \mathcal O (\mm^d p + (\log \mm ) p^{3d} )
    \qquad \mbox{and}\qquad
    \mathcal O (\mm^d p + \mm p^{2d} + (\log \mm ) p^{3d})
\]
for smoothing in the V-cycle and the W-cycle, respectively.
The full complexity including the costs for the exact coarse-grid solver and the
intergrid transfers is asymptotically the same.
Under mild assumptions on the relation between $p$ and $\mm$, again the overall effort is
asymptotically not higher than that for one application of the stiffness matrix.

%%%%%%%%%%%%%%%%%%%%%%%%%%%%%%%%%%%%%%%%%%%%%%%%%%%%%%%%%%%%%%%%%%%%%%%%%%%%%%%%

\section{Numerical experiments}
\label{sec:experiments}

\subsection{Experiments for the model problem}

We solve the problem \eqref{eq:vf}, i.e.,
\[
    -\Delta u + u = f  \quad \text{in } \Omega=(0,1)^d,
    \qquad
    \partial_n u = 0   \quad \text{on } \partial\Omega
\]
for $d=1,2,3$ with the right-hand side
\begin{equation}\label{eq:def:f}
    f( x ) = d \pi^2 \prod_{j=1}^d \sin(\pi (x_j + \tfrac12)).
\end{equation}
We perform a (tensor product) B-spline discretization using equidistant
knot spans and maximum-continuity splines for varying spline degrees $p$.
We refer to the coarse discretization with only one single interval as level $\ell=0$
and perform uniform, dyadic refinement to obtain the finer discretization levels $\ell$
with $2^{\ell d}$ elements and $h_\ell=2^{-\ell}$.

We set up a V-cycle multigrid method as described in Section~\ref{sec:mg} and
using on each level the proposed smoother \eqref{eq:smoother} as constructed in
Section~\ref{sec:construction}.
We always use one pre- and one post-smoothing step with $\tau = 1$. The parameter
$\sigma$ was chosen as
$\tfrac1{0.09}\,h^{-2}$ in 1D,
$\tfrac1{0.18}\,h^{-2}$ in 2D, and
$\tfrac1{0.19}\,h^{-2}$ in 3D.
In each test, the coarsest grid was chosen in such a way that the spaces
$S_0$ on each higher level are non-empty, i.e., such that the smoother is well-defined.
We perform tests both using the V-cycle multigrid method and a conjugate gradient
solver preconditioned with one V-cycle.
The iteration numbers required to reduce the $\ell^2$-norm of the initial residual
by a factor of $10^{-8}$ for the 1D, 2D and 3D problem are given
in Tables~\ref{tab:iter-1d}--\ref{tab:iter-3d}, respectively.

\begin{table}[htbp]
    \caption{Iteration numbers: unit interval (1D)}
    \label{tab:iter-1d}
    \centering
    \begin{tabular}{llrrrrrrrrrrrrr}
    \toprule
    & $\ell\; \diagdown\; p$
           & 2  & 3  & 4  & 5  & 6  & 7  & 8  & 9  & 10 & 11 & 12 & 13 & 14 \\
    \midrule   % V-cycle iteration
    \multirow{3}{*}{V-cycle}
           & 9      & 33 & 34 & 34 & 33 & 33 & 33 & 32 & 31 & 31 & 31 & 28 & 28 & 29 \\
           & 8      & 33 & 34 & 34 & 32 & 33 & 33 & 31 & 30 & 30 & 31 & 28 & 28 & 27 \\
           & 7      & 33 & 34 & 34 & 32 & 33 & 33 & 31 & 28 & 30 & 29 & 28 & 25 & 26 \\
    \midrule   % CG preconditioned with one V-cycle
    \multirow{3}{*}{PCG}
           & 9      & 13 & 13 & 13 & 13 & 13 & 13 & 13 & 13 & 12 & 12 & 12 & 12 & 12 \\
           & 8      & 13 & 13 & 13 & 13 & 13 & 13 & 12 & 12 & 12 & 12 & 12 & 12 & 11 \\
           & 7      & 13 & 13 & 13 & 13 & 13 & 12 & 12 & 12 & 12 & 11 & 11 & 11 & 11 \\
    \bottomrule
    \end{tabular}
\end{table}

\begin{table}[htbp]
    \caption{Iteration numbers: unit square (2D).}
    \label{tab:iter-2d}
    \centering
    \begin{tabular}{llrrrrrrrrr}
    \toprule
    & $\ell\; \diagdown\; p$
           & 2  & 3  & 4  & 5  & 6  & 7  & 8  & 9  & 10 \\
    \midrule   % V-cycle iteration
    \multirow{4}{*}{V-cycle}
           & 8      & 38 & 39 & 39 & 39 & 38 & 38 & 37 & 37 & 36 \\
           & 7      & 38 & 39 & 39 & 38 & 38 & 37 & 36 & 36 & 34 \\
           & 6      & 38 & 38 & 38 & 37 & 37 & 35 & 34 & 34 & 32 \\
           & 5      & 36 & 37 & 34 & 34 & 32 & 30 & 28 & 26 & 24 \\
    \midrule   % CG preconditioned with one V-cycle
    \multirow{4}{*}{PCG}
           & 8      & 14 & 14 & 14 & 14 & 14 & 14 & 14 & 14 & 13 \\
           & 7      & 14 & 14 & 14 & 14 & 14 & 14 & 14 & 13 & 13 \\
           & 6      & 14 & 14 & 14 & 14 & 14 & 13 & 13 & 13 & 12 \\
           & 5      & 14 & 14 & 13 & 13 & 13 & 12 & 11 & 11 & 10 \\
    \bottomrule
    \end{tabular}
\end{table}

\begin{table}[h!]
    \caption{Iteration numbers: unit cube (3D).}
    \label{tab:iter-3d}
    \centering
    \begin{tabular}{llrrrrrr}
    \toprule
    &$\ell\; \diagdown\; p$
                  &  2 &  3 &  4 &  5 &  6 &  7 \\
    \midrule
    \multirow{4}{*}{V-cycle}
      & 6             & 46 & 44 & 43 & 43 & 42 & 41 \\
      & 5             & 44 & 43 & 42 & 39 & 38 & 35 \\
      & 4             & 39 & 36 & 32 & 29 & 25 & 23 \\
      & 3             & 30 & 42 & 18 & 22 & 12 & 17 \\
    \midrule   % CG preconditioned with one V-cycle
    \multirow{4}{*}{PCG}
      & 6             & 17 & 16 & 15 & 15 & 15 & 15 \\
      & 5             & 17 & 16 & 15 & 15 & 14 & 13 \\
      & 4             & 14 & 16 & 13 & 14 & 11 & 12 \\
      & 3             & 12 & 13 &  9 & 10 &  7 &  8 \\
    \bottomrule
    \end{tabular}
\end{table}

The method was implemented in C++ based on the G+SMO
library\footnote{\url{http://www.gs.jku.at/gismo}} which is developed in the
framework of the National Research Network ``Geometry + Simulation'' at
Johannes Kepler University, Linz.

We observe that the iteration numbers are robust with respect to both the
discretization level $\ell$ (and thus $h$) and the spline degree $p$.  They do
increase with the space dimension $d$, but this dependence, which we have not
fully analyzed, appears to be relatively mild.
In particular, the 2D iteration numbers are significantly lower than those
obtained using the boundary-corrected mass smoother in \cite{HTZ:2016}.

\subsection{Experiments for non-trivial computational domains}

We perform experiments with varying, matrix-valued diffusion coefficients
on the non-trivial geometries shown in Fig.~\ref{fig:geo}.
The geometry map for the quarter annulus in the two-dimensional example is
described exactly with NURBS, that for the three-dimensional object with
B-splines.
On these objects, we solve
\[
    -\operatorname{div}(A(x) \nabla u(x)) = f(x) \quad  \qquad \text{in } \Omega
\]
with Dirichlet boundary conditions $g(x)$ on $\Gamma_D$ as indicated in
Fig.~\ref{fig:geo} and homogeneous Neumann boundary conditions on the
remaining part of the boundary.
Furthermore, $f$ is given by~\eqref{eq:def:f} and the diffusion coefficient is given by
\[
    A^{\text{(2D)}}(x) =
     \begin{pmatrix}
            1 + x_1^2   &   -x_1 x_2   \\
            -x_1 x_2      &  1 + x_2^2
     \end{pmatrix},
     \quad
    A^{\text{(3D)}}(x) =
     \begin{pmatrix}
             1 + x_1^2         &   -\tfrac13 x_1 x_2  &   -\tfrac13 x_1 x_3  \\
            -\tfrac13 x_1 x_2  &    1 + x_2^2         &   -\tfrac13 x_2 x_3  \\
            -\tfrac13 x_1 x_3  &   -\tfrac13 x_2 x_3  &    1 + x_3^2    \\
     \end{pmatrix}.
\]

\begin{figure}[htbp]
  \begin{center}
  \includegraphics[width=.3\textwidth]{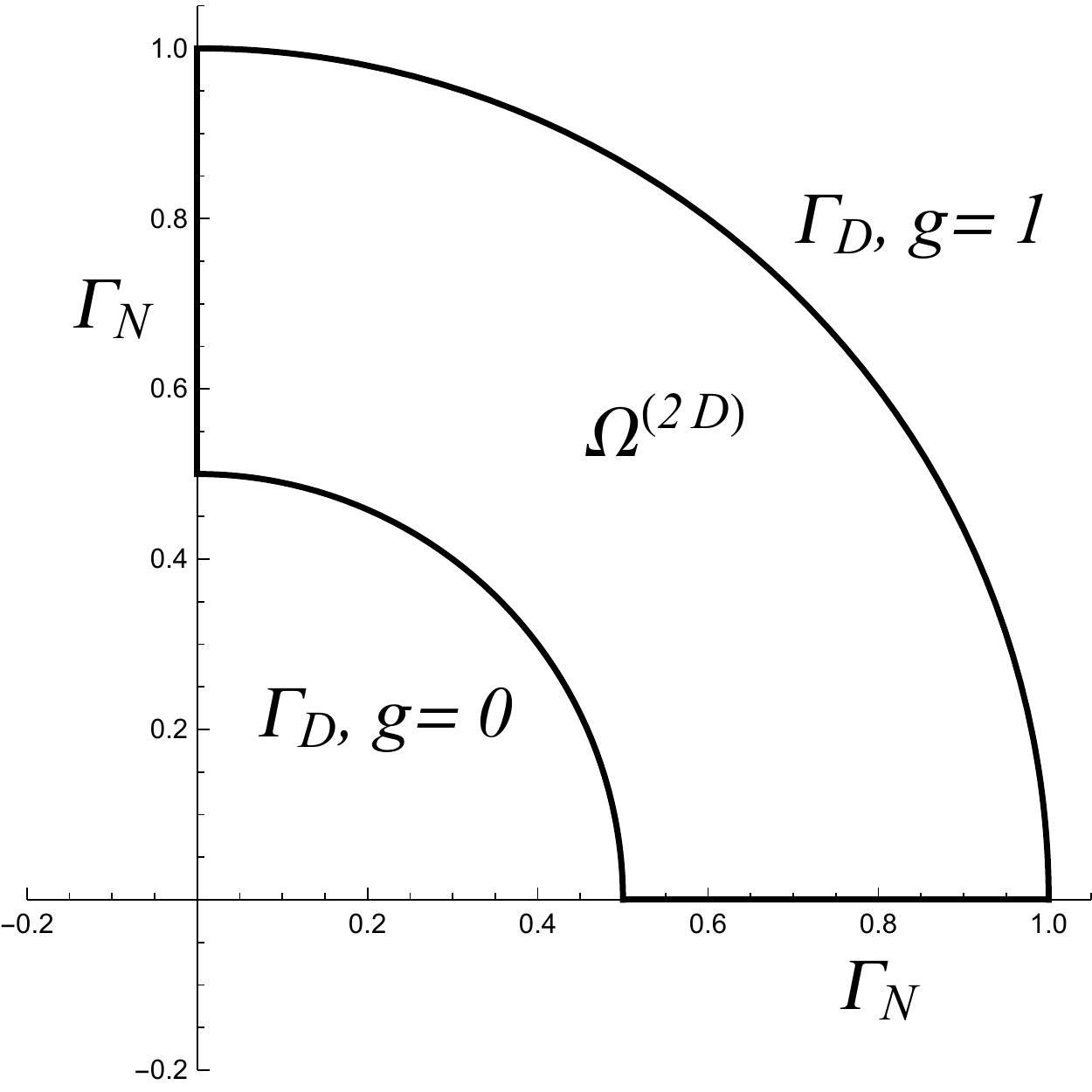}
  \qquad
  \includegraphics[width=.3\textwidth]{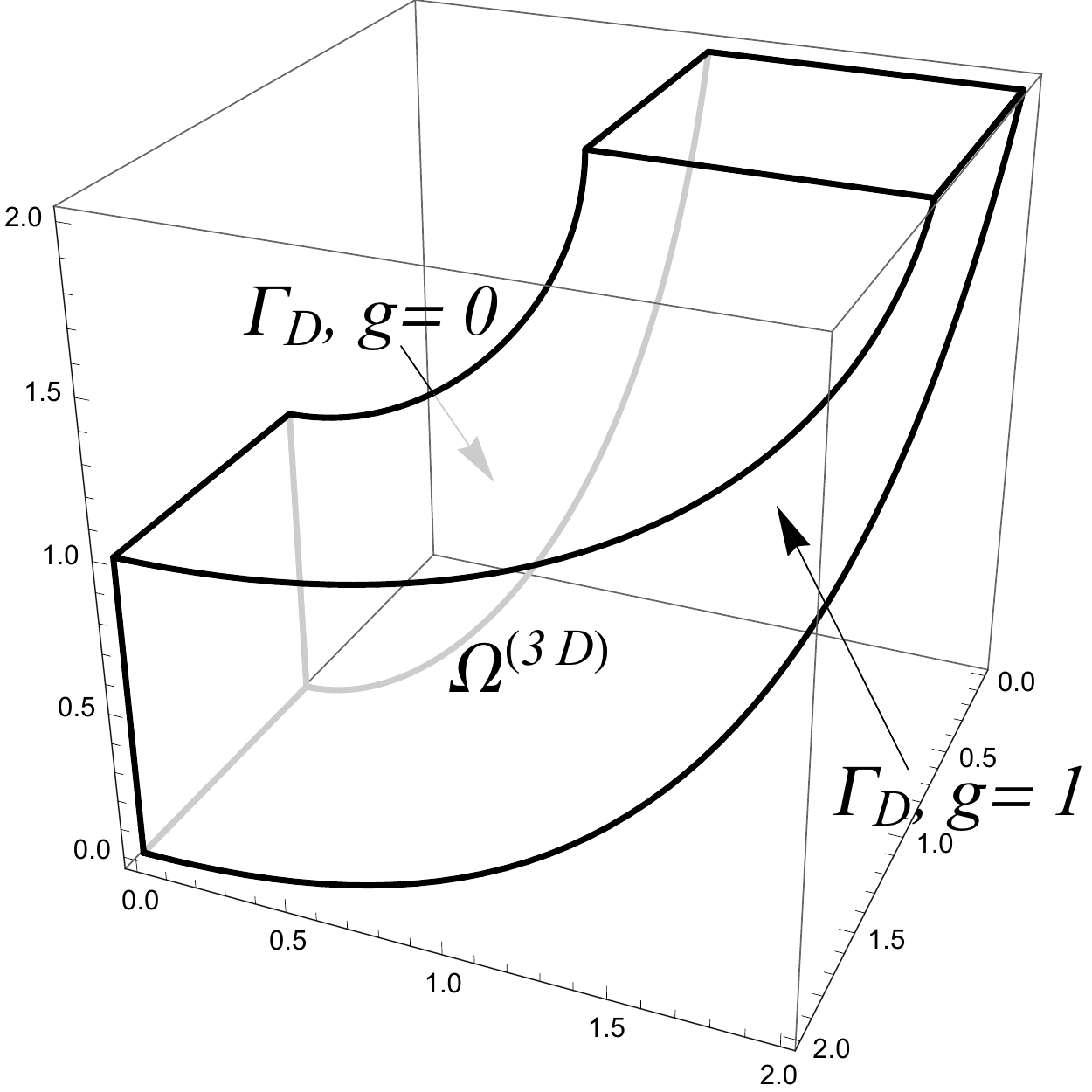}
  \end{center}
  \caption{Computational domains for 2D and 3D example.}
  \label{fig:geo}
\end{figure}

Table~\ref{tab:iter-2d-3d-qa} gives the iteration numbers
for a conjugate gradient method, preconditioned with one V-cycle of the
proposed multigrid solver, where the multigrid solver was set up as solver for the
model problem $-\Delta u+u=f$ on the parameter domain.

\begin{table}[htbp]
    \caption{Iteration numbers for the nontrivial 2D (top) and
        3D (bottom) domains as shown in Fig.~\ref{fig:geo}.}
    \label{tab:iter-2d-3d-qa}
    \begin{center}
     \begin{minipage}{0.7\textwidth}
    \begin{tabular}{lrrrrrrrrr}
    \toprule
    $\ell\; \diagdown\; p$
                  & 2  & 3  & 4  & 5  & 6  & 7  & 8  & 9  & 10 \\
    \midrule         % V-cycle iteration
           8      & 53 & 55 & 56 & 56 & 55 & 55 & 55 & 54 & 54 \\
           7      & 52 & 53 & 54 & 53 & 53 & 52 & 51 & 50 & 51 \\
           6      & 47 & 50 & 50 & 48 & 48 & 48 & 46 & 46 & 45 \\
           5      & 43 & 45 & 45 & 44 & 44 & 41 & 41 & 40 & 41 \\
    \bottomrule
    \end{tabular}
    \end{minipage}

    \vspace{2ex}
    \begin{minipage}{0.7\textwidth}
    \begin{tabular}{lrrrrrr}
    \toprule
    $\ell\; \diagdown\; p$
                  &  2 &  3 &  4 &  5 &  6 &  7 \\
    \midrule
      5             & 87 & 90 & 91 & 90 & 89 & 90 \\
      4             & 73 & 76 & 76 & 79 & 81 & 83 \\
      3             & 55 & 61 & 66 & 67 & 72 & 75 \\
    \bottomrule
    \end{tabular}
    \end{minipage}
    \end{center}
\end{table}

Obviously, the condition number of
the preconditioned system depends only on the geometry transformation, the
diffusion coefficient and on the contraction number of the multigrid method (as
a solver for the model problem on the parameter domain). All of these
quantities are independent of the grid size and the polynomial degree $p$. This
is reflected in the numerical results, which are robust in those two
parameters.

%%%%%%%%%%%%%%%%%%%%%%%%%%%%%%%%%%%%%%%%%%%%%%%%%%%%%%%%%%%%%%%%%%%%%%%%%%%%%%%%

\section*{Appendix}

The aim of this section is to prove Lemma~\ref{lemma:approx:dd}, an
approximation result for the coarse spline space Galerkin projector in $d$
dimensions.  It was shown in \cite{HTZ:2016} for $d=1$ and $d=2$, and here we
extend it to arbitrary dimensions by induction.

Before we give the proof, we need an auxiliary lemma which is a variant of the
Aubin-Nitsche duality argument in a finite-dimensional Hilbert space $V$.  By
the choice of a suitable basis, we can identify $V$ with $\mathbb R^n$, and
operators $A$ on $V$ with matrices.  We use this matrix representation
implicitly in the following, and operations like $A^{1/2}$ and $A^\top$ are to
be understood in the matrix sense.

\begin{lemma}\label{lemma:dual}
    Let $A$ and $M$ be self-adjoint and positive definite linear operators on $V$,
    $T : V \to W \subset V$ an $A$-orthogonal projector, and $\theta > 0$.  Then,
    the statements
    \begin{equation}~\label{eq:dual:1:2}
        \|T u\|_M \le \theta \|u\|_A \quad\forall u \in V
        \qquad\text{and}\qquad
        \|T u\|_A \le \theta \|u\|_{AM^{-1}A} \quad\forall u \in V
    \end{equation}
    are equivalent.
\end{lemma}
\begin{proof}
    We first observe that the statements in~\eqref{eq:dual:1:2} are equivalent to
    \begin{equation}~\label{eq:dual:3}
        \|M^{1/2}T A^{-1/2}\| \le \theta
        \qquad\text{and}\qquad
        \|A^{1/2}T A^{-1}M^{1/2}\| \le \theta,
    \end{equation}
    respectively.
    Since $T$ is self-adjoint in the scalar product $(\cdot,\cdot)_A$,
    $A T = T^\top A$ and further
    \begin{equation}~\label{eq:dual:5}
        T A^{-1} = A^{-1} T^\top
    \end{equation}
    hold.
    Using \eqref{eq:dual:5} as well as the self-adjointness of $M$ and $A$, we obtain
    \begin{equation*}
        \|M^{1/2}T A^{-1/2} \| = \|M^{1/2}A^{-1}T^\top A^{1/2} \| =\|(A^{1/2}T A^{-1}M^{1/2})^\top \| = \|A^{1/2}T A^{-1}M^{1/2}\|.
    \end{equation*}
    This proves that the two statements in \eqref{eq:dual:3} and, consequently, those in~\eqref{eq:dual:1:2} are equivalent.
\end{proof}

\begin{proof}[Proof of Lemma~\ref{lemma:approx:dd}]
    Within this proof, we denote the dimensions explicitly and
    use a recursive representation,
    \begin{alignat*}{3}
        M_1 &:= M,                 \qquad &  A_1 &:= K + M, \\
        M_d &:= M_{d-1} \otimes M, \qquad &  A_d &:= A_{d-1} \otimes M + M_{d-1} \otimes K.
    \end{alignat*}
    Furthermore we let $T_d$ denote the $A_d$-orthogonal projector into $(S_c)^d$.

    In \cite{HTZ:2016}, the desired result was proved for $d=1$, namely,
    \begin{equation}\label{app:1}
        \| (I-T_1) u \|_{M_1} \le c h \|u\|_{A_1}	\qquad \forall u\in S.
    \end{equation}
    By Lemma~\ref{lemma:dual}, this is equivalent to
    \begin{equation}
        \label{app:2}
        \|  (I-T_1) u \|_{A_1 } \le ch \|u\|_{A_1 M_1^{-1} A_1}
        \qquad \forall u\in S.
    \end{equation}
    Stability of the $A_1$-orthogonal projector means that
    \begin{equation}
        \label{app:3}
        \| (I-T_1) u \|_{A_1} \le  \|u\|_{A_1}	\qquad \forall u\in S.
    \end{equation}
    We now show the desired result using induction. Assume that we have already
    shown
    \begin{equation}\label{app:4}
        \| (I-T_{d-1}) u \|_{M_{d-1}} \le c h \|u\|_{A_{d-1}}
        \qquad \forall u\in S^{d-1}
    \end{equation}
    for some $d > 1$.
    Using Lemma~\ref{lemma:dual}, this implies
    \begin{equation}
        \label{app:2a}
        \|  (I-T_{d-1}) u \|_{A_{d-1} } \le ch \|u\|_{A_{d-1} M_{d-1}^{-1} A_{d-1}}
        \qquad  \forall u\in S^{d-1}.
    \end{equation}
    Stability of the $A_{d-1}$-orthogonal projector means that
    \begin{equation}
        \label{app:3a}
        \| (I-T_{d-1}) u \|_{A_{d-1}} \le  \|u\|_{A_{d-1}}
        \qquad  \forall u\in S^{d-1}.
    \end{equation}
    Using equations \eqref{app:1}--\eqref{app:3a} and the fact that the
    operator norm of a tensor product is the product of the individual operator
    norms, we obtain for all $u\in S^d$
    \begin{align*}
        &\|  (I- T_{d-1})\otimes(I- T_1) u \|_{A_{d-1}\otimes M_1+M_{d-1} \otimes A_1}
        \le c h \|u\|_{A_{d-1}\otimes A_1}, \\
        &
        \| (I- T_{d-1})\otimes I u \|_{A_{d-1}\otimes M_1+M_{d-1} \otimes A_1}
        \le c h \|u\|_{A_{d-1} M_{d-1}^{-1} A_{d-1} \otimes M_1 + A_{d-1}\otimes A_1},
        \\
        &
        \| I\otimes (I- T_1) u \|_{A_{d-1}\otimes M_1+M_{d-1} \otimes A_1}
        \le c h \|u\|_{A_{d-1}\otimes A_1 + M_{d-1} \otimes A_1 M_1^{-1} A_1}.
    \end{align*}
    Since
    $%\[
        I- T_{d-1}\otimes T_1
        =
        (I- T_{d-1})\otimes I +I\otimes (I- T_1) - (I- T_{d-1})\otimes(I- T_1)
    $, %\]
    this implies using the triangle inequality
    \begin{multline*}
        \|  (I- T_{d-1}\otimes T_1) u \|_{A_{d-1}\otimes M_1+M_{d-1} \otimes A_1}
        \le c h \|u\|_{A_{d-1} M_{d-1}^{-1} A_{d-1} \otimes M_1
        + A_{d-1} \otimes A_1 + M_{d-1} \otimes A_1 M_1^{-1} A_1}.
    \end{multline*}
    As the norm on the left-hand side is bounded from below by
    $\|\cdot\|_{A_d}$ and the norm on the right-hand side is bounded from
    above by $c\|\cdot\|_{A_d M_d^{-1} A_d}$, we further obtain
    \[
        \|  (I- T_{d-1}\otimes T_1) u \|_{A_d} \le c h \|u\|_{A_d M_d^{-1} A_d}
        \qquad \forall u\in S^d.
    \]
    Both $T_{d-1}\otimes T_1$ and $T_d$ are projectors
    into $(S_c)^d$.  Since the latter projector produces the best approximation
    in the $A_d$-norm, we have
    \[
        \|  (I- T_d) u \|_{A_d} \le c h \|u\|_{A_d M_d^{-1} A_d}
        \qquad \forall u\in S^d,
    \]
    which, by Lemma~\ref{lemma:dual}, is equivalent to the desired result
    \[
        \|  (I- T_d) u \|_{M_d} \le c h \|u\|_{A_d}
        \qquad \forall u\in S^d.
    \]
\end{proof}

%%%%%%%%%%%%%%%%%%%%%%%%%%%%%%%%%%%%%%%%%%%%%%%%%%%%%%%%%%%%%%%%%%%%%%%%%%%%%%%%

\section*{Acknowledgments}

We gratefully acknowledge the discussions with Ludmil Zikatanov (Penn State
University) which were instrumental in developing some of the ideas underlying
this work.

%%%%%%%%%%%%%%%%%%%%%%%%%%%%%%%%%%%%%%%%%%%%%%%%%%%%%%%%%%%%%%%%%%%%%%%%%%%%%%%%

\bibliographystyle{siamplain}
\bibliography{references}

\end{document}